\documentclass{IEEEtran}
\usepackage{cite}
\usepackage{amsmath,amssymb,amsfonts}
\usepackage{algorithmic}
\usepackage{graphicx}
\usepackage{textcomp}
\usepackage{algorithm,algorithmic}
\usepackage{hyperref}
\usepackage{subfigure}

\newtheorem{assumption}{Assumption}

\newtheorem{remark}{Remark}
\newtheorem{theorem}{Theorem}
\newtheorem{proof}{Proof}

\hypersetup{hidelinks=true}
\usepackage {subcaption}
\usepackage{color}
\def\BibTeX{{\rm B\kern-.05em{\sc i\kern-.025em b}\kern-.08em
		T\kern-.1667em\lower.7ex\hbox{E}\kern-.125emX}}
\begin{document}
	\title{Distributed Optimal Consensus of Nonlinear Multi-Agent Systems: A Unified Approach for Leaderless and  Leader-follower}
	\author{Ziyuan Guo, Chuanzhi lv, Liping Zhang and Huanshui Zhang, \IEEEmembership{Senior Member, IEEE}
		\thanks{This work was supported by  the Original Exploratory Program Project of National Natural Science Foundation of China (62450004) and the Joint Funds of the National Natural Science Foundation of China (U23A20325).  }
		\thanks{Ziyuan Guo, Chuanzhi Lv, and Liping Zhang are with the College of
			Electrical Engineering and Automation, Shandong University of Science
			and Technology, Qingdao 266590, China (e-mail: skdgzy@sdust.edu.cn; lczbyq@sdust.edu.cn; lpzhang1020@sdust.edu.cn).}
		\thanks{Huanshui Zhang is with the College of Electrical Engineering and
			Automation, Shandong University of Science and Technology, Qingdao
			266590, China, and also with the School of Control Science and
			Engineering, Shandong University, Jinan, Shandong 250061, China (email: hszhang@sdu.edu.cn).}
	}
	
	\maketitle
	
	\begin{abstract}
		In this paper, the optimal consensus problem for general nonlinear multi-agent systems is studied, where both leaderless and leader-follower cases are considered in a unified framework. The key idea is to convert consensus problems into optimal control problems where the objective of each agent with nonlinear dynamics is to design the control input minimizing the global consensus cost function. Compared with the existing distributed consensus control for nonlinear multi-agent systems, we propose a distributed optimal consensus algorithm  based on the optimal control principle (OCP) method, and two enhanced algorithms are developed under the model predictive control (MPC) framework—these two algorithms demonstrate broader applicability when handling general nonlinear multi-agent systems. Moreover, the convergence and superlinear convergence rate of the proposed algorithms are rigorously analyzed. Numerical simulations demonstrate the effectiveness of the proposed algorithms.
	\end{abstract}
	
	
	\section{Introduction}
	\IEEEPARstart{O}{ver} the past few decades, consensus in multi-agent systems has emerged as a rapidly growing topic across various research fields due to its wide range of applications, such as agents formation \cite{bib1}, \cite{bib2}, sensor networks \cite{bib3}, distributed optimization \cite{bib4}, \cite{bib5}, and smart grids \cite{bib6}, \cite{bib7}. The key to the consensus for each individual in a multi-agent system lies in designing a distributed control protocol based solely on its own and neighboring agents' available information that ensures state or output consensus. From the perspective of leader existence, such control problems can be further categorized into leaderless consensus control \cite{bib8} and leader-following consensus control, as seen in \cite{bib9}, \cite{bib10}. 
	
	Leaderless consensus has attracted significant research attention due to its fully decentralized nature and high scalability. It typically refers to scenarios where, in the absence of a pre-designated leader, all agents in the system achieve a common state or output through distributed control protocols. Over the past few years, the leaderless consensus problem has been investigated by many scholars. For linear multi-agent systems, extensive research has been conducted in this area, as evidenced by \cite{bib11,bib12,bib13,bibn1}. Although substantial research achievements have been made for linear systems, practical engineering problems predominantly involve nonlinear system models, such as robotic manipulator dynamics, power grid transient stability, and aircraft stall dynamics, etc. Adopting linearization techniques such as Taylor series expansion, feedback linearization, or incremental linearization to handle actual nonlinear multi-agent systems may lead to the destruction of the original system's structural characteristics, while also increasing the conservatism of the resulting control protocols. Consequently, nonlinear multi-agent systems have been the focus of substantial scholarly investigation. In \cite{bib14}, the author propose a distributed leaderless consensus algorithms for networked Euler-Lagrange systems. Regarding the consensus problem of multi-agent systems with Lipschitz nonlinear dynamics, \cite{bib15} proposes an adaptive consensus protocol that does not rely on any global information. \cite{bib16} investigates the leaderless consensus problem for first-order nonlinear multi-agent systems under jointly connected communication topologies. In \cite{bib17}, a distributed adaptive consensus control protocol is designed for leaderless strict-feedback nonlinear multi-agent systems, with rigorous stability analysis provided. While extensive research has been conducted on leaderless consensus control for nonlinear multi-agent systems, the development of distributed control protocol for general-form nonlinear systems remains unsatisfactory. Most existing studies focus on designing distributed consensus control protocols for specific classes of nonlinear multi-agent systems. It must be acknowledged that this approach certainly offers advantages—such as facilitating Lyapunov-based stability analysis—yet simultaneously introduces practical implementation difficulties. In reality, many real-world systems cannot be matched with the specialized nonlinear multi-agent systems assumed in these studies, indicating that the applicability of these methods is relatively limited. Moreover, excessive focus on algorithmic adaptivity often leads to overly complex and conservative consensus protocols, posing implementation challenges in practical engineering applications. 
	
	On the other hand, leader-follower consensus control has also garnered considerable attention, particularly in scenarios requiring well-defined mission objectives or hierarchical control architectures. The main focus of leader-follower consensus control is to enable multi-agent systems to rapidly and stably achieve a predefined global consensus state or track specific dynamic targets under the leader's guidance, through distributed consensus  protocols. While extensive research has been conducted on leader-follower consensus control for linear systems \cite{bib18,bib19,bib20,bib21}, challenges such as model mismatch, robustness to uncertainties, and scalability to complex dynamics persist—motivating growing interest in nonlinear system extensions. In \cite{bib17}, a class of strict-feedback nonlinear systems is investigated. To design distributed adaptive leader-follower consensus control protocols, a local compensatory variable is generated based on signals collected from neighboring agents. For a class of high-order nonlinear systems with time-varying references, \cite{bib22} introduces local estimators for reference trajectory bounds and individual agent filters, proposing a  smooth distributed adaptive control protocol based on backstepping to address the leader-following consensus problem. For a class of uncertain nonlinear second-order systems, \cite{bib23} proposes an approach combining consensus protocols with dynamical neural networks (specifically differential neural networks) to approximate follower agents' modeling uncertainties and external disturbances. For leader-follower nonlinear high-order multi-agent systems, a fully distributed consensus tracking control scheme is developed in \cite{bib24}, employing an adaptive-gain fixed-time command filter to ensure filter error convergence. Like leaderless consensus, leader-follower frameworks face limitations: most protocols target specific nonlinear forms, whereas general nonlinear systems better match real-world applications. Thus, developing consensus control protocols for general nonlinear systems is more aligned with practical needs. For instance, in military cooperative operations involving multi-echelon joint strike formations across naval, land, and air domains, the nonlinear characteristics vary significantly among different platform models. Moreover, when addressing emergent scenarios—such as incorporating additional nonlinear agents of heterogeneous types into the original system to follow the leader's signals—more general consensus control protocols demonstrate superior adaptability.
	
	To develop effective distributed consensus control protocols for general consensus problems, optimal control theory serves as a powerful theoretical framework. In our previous work, we investigated the consensus problem for linear systems and derived an analytical form of the optimal consensus controller\cite{bibn3}, \cite{bibn4}. However, the design of the error observer it relies upon is relatively complex, and obtaining an analytical solution for optimal control in nonlinear systems remains challenging. Recently, we successfully applied optimal control theory to optimization problems and proposed an OCP method with superlinear
	convergence rate \cite{bib25}, \cite{bib26}, which prompted us to consider the possibility of applying this method to solve distributed consensus problems. Inspired by the aforementioned work and \cite{bib27}, \cite{bib28}, we study the consensus problem of nonlinear multi-agent systems from the viewpoint of the optimal control problem. To be specific, we design a global cost function to accomplish the consensus task. Our objective is to design control inputs for each nonlinear agent that minimizes this  cost function. Based on the maximum principle, a new optimal consensus algorithm is derived and two enhanced algorithms are developed under the MPC framework. In the design of consensus control for multi-agent systems, leaderless and leader-follower protocols are typically designed separately. Protocols designed for the leaderless case cannot function effectively on graphs for leader-follower scenarios where the directed spanning tree is not strongly connected. Conversely, protocols designed for leader-follower (relying on leader-driven dynamics) will either fail in the absence of a leader, or require additional mechanisms to designate or elect a leader. Unlike \cite{bib14}, \cite{bib23} that address only leaderless or leader-follower configurations separately, our algorithms achieve unified applicability to both scenarios. Besides, it is different from the above-mentioned method of specific classes of nonlinear multi-agent systems that our algorithms are suitable for more general nonlinear systems and the control input sequence can be modified online to enhance the stability of the closed-loop system under the MPC framework. Compared with the existing results of distributed consensus algorithms from the viewpoint of the optimal control theory \cite{bib11,bib29,bib30}, the algorithms proposed in this paper achieve significantly faster convergence rates while maintaining simpler implementation requirements.
	
	The paper is organized as follows. In Section II, the control
	problem is formulated and some necessary preliminaries are provided. In Section III, the consensus control schemes and specific algorithms are given. In Section IV, a convergence analysis of the proposed algorithms is conducted. In Section V, three simulation examples are shown to illustrate the effectiveness of the proposed algorithms and
	finally the paper is concluded in Section VI.
	
	{\bf Notation}: $A^T$ stands for the transposition of
	matrix $A$. $\mathbb{R}^n$ denotes the $n$-dimensional real vector space. $I$ is a unit matrix with appropriate dimension. For a symmetric matrix $M$, $M > 0(\geq 0)$ means
	that $M$ is a positive definite (positive semi-definite) matrix. $\rho(M)$ is the spectral radius of $M$.  $\|\cdot\|$ denotes 2-norm of vectors and $\|\cdot\|_m$ is the induced norm of matrices. 
	\section{PROBLEM FORMULATION}
	\subsection{Multi-agent system model}
	Consider a multi-agent system consisting of $n$ nonlinear agents:
	\begin{equation}
		\begin{split}
			&x_{i}(k+1) = f_i(x_{i}(k),u_{i}(k)),\\ 
			&x_i(0)=a_i,\ i \in\{1,2,\dots,n\} \label{nonagents}
		\end{split}
	\end{equation}
	where $f_i\in\mathbb{R}^p$ is a given deterministic function, $x_i(k) \in \mathbb{R}^p$ is the state variable of $i$-th agent, $u_i(k) \in \mathbb{R}^m$ is the control variable of $i$-th agent, and $x_i(0)$ is the given initial state.
	
	\begin{remark}
		The general nonlinear system formulation can encompass a wide range of practical nonlinear systems, such as \cite{bib17}, \cite{bib23}. Notably, we do not require identical agent dynamics across the multi-agent system, thereby significantly enhancing the generality of our results.
	\end{remark}
	
	\subsection{Graph theory}
	The communication graph is denoted by a directed graph $\mathcal{G=(V,E,A)}$, where $\mathcal{V}=\{1,2, \ldots ,n\}$ is a set of vertices(nodes), $n$ is the number of agents satisfying $n \geq 2$, $\mathcal{E} \subseteq \mathcal{V} \times \mathcal{V}$ is the set of edges, and the
	weighted matrix $\mathcal{A} = (a_{ij})_{n\times n}$ is a non-negative matrix
	for adjacency weights of edges, $a_{ij} \neq 0$ if and only if {$(i, j) \in \mathcal{E}$}. The graph $\mathcal{G}$ is said to be balanced if the sum of the interaction weights from {agent $j$} to agent $i$ are equal, i.e., ${\textstyle\sum_{j=1}^{n}a_{ij}}={\textstyle\sum_{j=1}^{n}a_{ji}}, \forall i\in \mathcal{V}$. Moreover, $\mathcal{N}_i = \{j \in \mathcal{V}|{(i, j)} \in \mathcal{E}\}$ is used to represent the neighbor set { of agent $i$}. For the topology $\mathcal{G}$, a path of length $r$ from node $i_1$ to node $i_{r+1}$ is a sequence of $r + 1$ distinct nodes ${\{i_1\ldots , i_{r+1}\}}$ such that $(i_q, i_{q+1}) \in \mathcal{E}$ for $q = 1, \ldots ,r$. A digraph has a spanning tree,
	if there is an agent called root, such that there is a directed path
	from the root to each other agent in the graph. If there exists a
	directed path between any two distinct nodes in directed graph
	$\mathcal{G}$, the graph is said to be strongly connected.
	\subsection{Control objectives}
	The control objectives of this paper are to design distributed optimal consensus control protocols for all  nonlinear agents (\ref{nonagents}) under the directed graph condition such that:
	\begin{itemize}	
		\item
		For the leaderless case, all agents states reach a consensus through the optimal consensus control protocols, i.e., $\textstyle \lim_{k \to \infty} (x_i(k)-x_j(k))=0$.
		\item
		For the leader-follower case, all agents states reach a consensus by tracking a common desired trajectory (leader's expected trajectory) $x_l(k)$  asymptotically, i.e., $\textstyle \lim_{k \to \infty} (x_i(k)-x_l(k))=0$.
	\end{itemize}
	\section{distributed optimal consensus}
	In this section, the optimal consensus control problem within a finite control period is formulated and distributedly solved using the proposed new algorithms.
	\subsection{Finite-time leaderless optimal consensus}
	To achieve the leaderless optimal consensus control objective, a necessary assumption is imposed.
	
	\begin{assumption}\label{asmpt1}
		The directed graph $\mathcal{G}$ is strongly connected. 	
	\end{assumption}
	
	We consider a finite control horizon $N$ and define the control period from $k=0$ to $k=N$. Here, $u_i$ is the control input sequence for the $i$-th nonlinear agent. Inspired by the optimal consensus problem for linear systems in \cite{bib31}, for a nonlinear multi-agent system, we consider the global cost function as follows:
	\begin{equation}
		\begin{aligned}
			J_1=&\sum\limits_{k = 0}^N \bigg(\sum\limits_{i = 1}^n\big(\sum\limits_{j \in \mathcal N_{i}}(x_i(k)-x_j(k))^{T}Q_{ij}(x_i(k)-x_j(k))\\
			&+ u_i(k)^\mathrm {T} R_{i}u_i(k) \big)\bigg)+\sum\limits_{i = 1}^n\big(\sum\limits_{j \in \mathcal{N}_i }(x_i(N+1)\\
			&-x_j(N+1))^{T}D_{ij}(x_i(N+1)-x_j(N+1))\big), \label{costf}
		\end{aligned}
	\end{equation}
	where $Q_{ij}\geq0$, $F_{ij}\geq0$ and $R_i>0$ are weighted matrices.
	
	For the leaderless case, the optimal consensus control problem is designing a distributed protocol that derives the optimal control input sequence $u_i$ for each nonlinear agent to minimize the cost
	(\ref{costf}) subject to (\ref{nonagents}), which can be formulated as follows.
	
	\textbf{Problem 1}: The finite-time optimal consensus control problem for the leaderless case:
	\begin{equation}
		\begin{array}{l}
			\mathop {\rm minimize}\limits_{U} J_1. \\
			\mathop{s.t.\ (\ref{nonagents})}
		\end{array}  \label{minJ}
	\end{equation}

	Following from \cite{bib32}, applying Pontryagin's maximum principle to the system (\ref{nonagents}) with the cost function (\ref{costf}), the following costate equations and equilibrium conditions are obtained:
	\begin{equation}
		0= R_iu_i(k) + \lambda_i(k+1)\frac{\partial f_i(x_{i}(k),u_{i}(k))}{\partial u_{i}(k)},\ \ \ \ \ \ \ \ \label{ph}
	\end{equation}
	\begin{equation}
		{\lambda_i(k)} = \sum\limits_{j \in \mathcal N_i}Q_{ij}(x_i(k)-x_j(k)) +\lambda_i(k+1)\frac{\partial f_i(x_{i}(k),u_{i}(k))}{\partial x_{i}(k)}, \label{costate}
	\end{equation}
	with the terminal value ${\lambda_i(N+1)} = \sum_{j \in\mathcal Ni}D_{ij}(x_i(N+1)-x_j(N+1))$, and $i \in \{1,2,\dots,n\}$.

	Observing the above equations (\ref{ph})-(\ref{costate}), we find an interesting fact that they can be solved in a distributed manner. From the perspective of the $i$-th agent, it is evident that it only needs the states $x_j(k)$ of its neighbors. 
	
	\begin{remark}
		The derived forms of (\ref{ph})-(\ref{costate}) play a pivotal role in designing subsequent distributed consensus algorithms, as $\lambda_i$ depends exclusively on neighboring states $x_j(k)$. The system's nonlinearity prevents obtaining a closed-form optimal solution for the cost function (hence our later development of an MPC-based enhancement algorithm). Notably, even for linear multi-agent systems, the optimal feedback controller for the global cost function would require complete global error information. Consequently, our primary focus centers on developing efficient distributed methods for solving (\ref{ph})-(\ref{costate}).
	\end{remark}
	
	Equations (\ref{ph})-(\ref{costate}) can be numerically solved using the method of successive approximations (MSA) \cite{bib33},\cite{bib34}, but its convergence rate is relatively slow. The slow convergence rate poses challenges for real-time algorithm deployment, necessitating the development of faster computational methods to address these limitations. In our previous work, we successfully applied optimal control theory to optimization problems and achieved some results \cite{bib25}, \cite{bib26}, \cite{bib35}, which prompted us to consider the possibility of applying this theory to solve distributed optimal consensus control problems, as these problems are also a special class of optimization problems. To improve computational efficiency, an OCP method has been developed in \cite{bib25}, \cite{bib26}, as outlined below.
	\begin{equation}
		\begin{split}
			{u_i^{r+1}} &= u_i^r-d_i^r(u_i^r)\\
			d_i^l(u_i^r)&= (G+\nabla^2 J(u_i^r))^{-1}\big(\nabla J(u_i^r)+Gd_i^{l-1}(u_i^r)\big), \\
			d_i^0(u_i^r)&= (G+\nabla^2 J(u_i^r))^{-1}\nabla J(u_i^r).\label{cs_ocp}
		\end{split}
	\end{equation}
	where $l=1,\dots,r$ means the required number of cycles and $G$ is an adjustable matrix, which can typically be configured as a scalar matrix for practical implementation. Notably, the algorithm encounters significant computational challenges when addressing dynamic problems, particularly problem (\ref{minJ}), stemming from the demanding computations of both $\nabla J(u_i^r)$ and $\nabla^2 J(u_i^r)$. For details on method (\ref{cs_ocp}), please refer to \cite{bib25}.
	
	In \cite{bib28}, we have obtained the computational methods for the $\nabla J(u)$ and $\nabla^2 J(u)$ of a single nonlinear system under a cost function, along with a rigorous theoretical proof. It is noteworthy that  (\ref{ph})-(\ref{costate}) can be regarded as an extension of the coupled forward-backward equations for a single nonlinear system in \cite{bib28}. From the perspective of individual agents, the method in \cite{bib28} can be readily adapted here to efficiently compute $\nabla J(u_i^r)$ and $\nabla^2 J(u_i^r)$. For convenience, we take a first-order nonlinear system as an example here to demonstrate the acceleration algorithm, which is also consistent with the description in [29]. The representation of this method in a nonlinear multi-agent system is as follows.
	
	Accelerated Calculation Method (ACM):\\
	For the $i$-th nonlinear agent, its $\nabla J(u_i^r)$ is
	\begin{equation}
		\nabla J(u_i^r) =\begin{bmatrix} \nabla J\big(u_i^r(0)\big) \\ \vdots \\ \nabla J\big(u_i^r(k)\big)  \\ \vdots \\ \nabla J\big(u_i^r(N)\big)  \end{bmatrix}, \label{F1}
	\end{equation}
	Moreover, the $\nabla J(u_i^r(k))$ and $\lambda_i^r(k)$ are derived as
	\begin{equation}
		\hspace{-2.7em} \nabla J(u_i^r(k)) =  R_iu_i^r(k) + \lambda_i^r(k+1)\frac{\partial f_i(x_{i}^r(k),u_{i}^r(k))}{\partial u_{i}^r(k)}. \label{f1}
	\end{equation}
	\begin{equation}
		{\lambda_i^r(k)} = \sum\limits_{j \in \mathcal N_{i}}Q_{ij}(x_i^r(k)-x_j^r(k)) +\lambda_i^r(k+1)\frac{\partial f_i(x_{i}^r(k),u_{i}^r(k))}{\partial x_{i}^r(k)}. \label{costate1}
	\end{equation}
	where $k \in \{0,1,\dots,N\}$. Thus, we have obtained $\nabla J(u_i^r)$.
	
	For the $i$-th nonlinear agent, its  $\nabla^2 J(u_i^r)$ is
	\begin{equation}
		\nabla^2 J(u_i^r) =\begin{bmatrix} \nabla^2 J_0\big(u_i^r\big) \\ \vdots \\ \nabla^2 J_s\big(u_i^r\big)  \\ \vdots \\ \nabla J_N\big(u_i^r\big)  \end{bmatrix}, \label{F1}
	\end{equation}
	where $s=0,1,\dots,N$, the $\nabla^2 J_s(u_i^r)$ is derived as
	\begin{equation}
		\nabla^2 J_s(u_i^r) =  [\frac{\partial P^s_i(0)}{\partial u_{i}^r(0)}\dots\frac{\partial P^s_i(k)}{\partial u_{i}^r(k)}\dots\frac{\partial P^s_i(N)}{\partial u_{i}^r(N)}], \label{F2JS}
	\end{equation}
	and the expression for $P_i^s(k)$ is given by
	\[
	P_i^s(k) = 
	\begin{cases}
		\begin{aligned}
			& \alpha_i^T(k+1)f_i(x_i^r(k),u_i^r(k)) \\
			&  +\beta_i^T(k)\big(\sum\limits_{j \in \mathcal N_{i}}Q_{ij}(x_i^r(k)-x_j^r(k))\\
			& +\lambda_i^r(k+1)\frac{\partial f_i(x_{i}^r(k),u_{i}^r(k))}{\partial x_{i}^r(k)}\big)
		\end{aligned},  k = s \\
		\begin{aligned}
			& \alpha_i^T(k+1)f_i(x_i^r(k),u_i^r(k)) \\
			& +\beta_i^T(k)\big(\sum\limits_{j \in \mathcal N_{i}}Q_{ij}(x_i^r(k)-x_j^r(k))\\
			&+\lambda_i^r(k+1)\frac{\partial f_i(x_{i}^r(k),u_{i}^r(k))}{\partial x_{i}^r(k)}\big)\\ 
			&+ L_i(x_i^r(k),u_i^r(k),\lambda_i^r(k+1)) 
		\end{aligned},  k \neq s \\ 
	\end{cases}
	\] 
	where $k \in \{1,2,\dots,N \}$, $L_i(x_i^r(s),u_i^r(s),\lambda_i^r(s+1))=\sum_{k=0}^{N}(R_iu_i^r(k) + \lambda_i^r(k+1)\frac{\partial f_i(x_{i}^r(k),u_{i}^r(k))}{\partial u_{i}^r(k)})$, and $L_i(x_i^r(k),u_i^r(k),\lambda_i^r(k+1))=0$ when $k\neq s$. $\alpha_i(k)$ and $\beta_i(k)$ in $P_i^s(k)$ can be calculate as 
	\begin{align}\left\{\begin{aligned}
			&\beta_i(k+1) = \frac{\partial P^s_i(k)}{\partial\lambda_{i}^r(k+1)},\\
			&\alpha_i(k) =\frac{\partial P^s_i(k)}{\partial x_{i}^r(k)} \label{ab},
		\end{aligned}\right.
	\end{align}
	where $\alpha_i(N+1)=0$ and $\beta_i(0)=0$. Thus, we can obtain $\nabla^2 J(u_i^r)$.
	
	\begin{remark}
		The ACM employs the maximum principle in two distinct phases. In the first phase, solving Problem 1 yields the gradient of cost function $J_1$ with respect to $u_i^r$. Subsequently, a novel optimal control problem is formulated to enable Hessian matrix computation through its solution. This dual application of the maximum principle facilitates an iterative computation process for both gradient and Hessian matrix via solution of the corresponding forward-backward differential equations (FBDEs). As mentioned in \cite{bib28}, this method establishes a unified framework that significantly enhances the efficiency of static optimization algorithms when applied to dynamic optimal control problems. Notably, the gradient and Hessian matrix expressions maintain identical forms to the one-dimensional case, regardless of whether the control variables $u_i$ are m-dimensional. Indeed, for optimizing the control variables $u_i$ over $N$ steps with m dimensions, the gradient can be expressed as an $N\cdot m$-dimensional column vector, and the Hessian matrix $\in\mathbb{R}^{Nm\times Nm}$. The relevant derivations and proofs can be found in detail in \cite{bib28}.
	\end{remark}
	
	With the ACM, the $\nabla J(u_i^r)$ and $\nabla^2 J(u_i^r)$ can be efficiently computed. For the $i$-th nonlinear agent, the finite-time leaderless optimal consensus  algorithm is formulated  as follows: 
	\begin{algorithm}[H]
		\renewcommand{\algorithmicrequire}{\textbf{Input:}}
		\renewcommand{\algorithmicensure}{\textbf{Output:}}
		\caption{Leaderless optimal consensus algorithm for the nonlinear multi-agent system}
		\label{alg1}
		\begin{algorithmic}[1]
			\STATE \textbf{Initialization}: Each agent checks whether it can communicate with its neighbors,  sets $x_i(0) \in {\mathbb{R}^p}$ and initialize the control sequence $u_i^0$. Set $t=0$ and $\epsilon>0$.
			
			\FOR {each agent $i=1,...,n$}
			\STATE  Calculate  $\{ x_i^r(0),\dots,x_i^r(N+1)\}$ by $u_i^r$ from (\ref{nonagents}).
			\STATE Send $\{ x_i^r(0),\dots,x_i^r(N+1)\}$ to its neighbors and receive $\{ x_j^r(0),\dots,x_j^r(N+1)\}$ from its neighbor $j$.	
			\ENDFOR	
			\STATE Each agent $i$ calculates $\nabla J(u_i^r)$ and $\nabla^2 J(u_i^r)$ using ACM. 
			\STATE If \( ||\nabla J(u_i^r)||<\epsilon \) for each agent $i$, execute Step 14; otherwise, continue.
			\STATE Each agent $i$ calculates ${d}^0_{i}(u_i^r)= (G+\nabla^2 J(u_i^r))^{-1}\nabla J(u_i^r)$.
			\FOR {$l=1,...,r$}
			\STATE $d_i^{l}(u_i^r)= (G+\nabla^2 J(u_i^r))^{-1}\big(\nabla J(u_i^r)+Gd_i^{l-1}(u_i^r)\big)$
			\ENDFOR
			\STATE Each agent update $u_i^{r+1} = u_i^r-d_i^r(u_i^r)$.
			\STATE Set $r=r+1$ and return to Step 3.
			\STATE Output the control sequence \( u_i^r \).
		\end{algorithmic}
	\end{algorithm}
	In Algorithm 1, communication with neighbors occurs only in Step 4, where each nonlinear agent receives state information from its neighbors and transmits its own state information to them,  while all other steps are computed independently by each nonlinear agent. Step 8 initializes the loop by computing the initial ${d}^0_{i}(u_i^r)$. After a finite number of \( r \) iterations, the finite-time leaderless optimal consensus control sequence \( u_i^r \) is ultimately obtained, as outlined in Step 14.
	
	Although \( N \) is finite in Algorithm 1, an excessively large value of \( N \) may adversely affect the algorithm's performance due to issues such as numerical precision limitations of the computer and inaccuracies in the model. Moreover, from the perspective of real-time computational efficiency, it is not advisable to select an overly large \( N \). Fortunately, MPC can mitigate these issues by solving a series of local optimal control problems online. The prediction horizon of MPC is typically not very large, and it updates control inputs based on the current system state. The numerous advantages of MPC have made it widely applicable in real-life scenarios and more information about MPC can be found in \cite{bib36}, \cite{bib37}. In practical implementation, each agent computes  $\nabla J(u_i^r)$ and $\nabla^2 J(u_i^r)$ using ACM  after communicating with its neighbors. Subsequently, the agent updates its own \( u_i^r \) by applying  (\ref{cs_ocp}). This process is repeated until the optimization converges. For notational convenience, we define:
	\begin{itemize}	
		\item
		$t_k$ as the current time instant.
		\item
		$[0, ..., N_p]$ as the finite prediction horizon.
		\item
		$x_i(\tau|t_k)$, where $\tau \in  [t_k, t_k+N_p]$, represents the predicted state at time $t_k$.
	\end{itemize}	
	For implementation details, please refer to Algorithm 2.
	\begin{algorithm}[H]
		\renewcommand{\algorithmicrequire}{\textbf{Input:}}
		\renewcommand{\algorithmicensure}{\textbf{Output:}}
		\caption{MPC-based leaderless optimal consensus algorithm for nonlinear multi-agent systems}
		\label{alg2}
		\begin{algorithmic}[1]
			\STATE Initialization: Each agent checks whether it can communicate with its neighbors,  sets $x_i(t_0) \in {\mathbb{R}^p}$ and initialize the control sequence $u_i^0=[u_i^0(t_0)^T,\dots,u_i^0(t_0+N_p)^T ]^T$. Set $\epsilon>0$.
			
			\FOR {each agent $i=1,...,n$}
			\WHILE{$t=t_k,k=0,1,2,...$}
			\item Set $r=0$ and  $x_i^0(t_k|t_k)=x_i(t_k)$. 		
			\ENDWHILE
			\STATE  Calculate  $\{ x_i^r(t_k|t_k),\dots,x_i^r(t_k+N_p|t_k)\}$ by $u_i^r$ from (\ref{nonagents}).
			\STATE Send $\{ x_i^r(t_k|t_k),\dots,x_i^r(t_k+N_p|t_k)\}$ to its neighbors and receive $\{ x_j^r(t_k|t_k),\dots,x_j^r(t_k+N_p|t_k)\}$ from its neighbor $j \in \mathcal N_i$. 
			\STATE Calculate $\nabla J(u_i^r)$ and $\nabla^2 J(u_i^r)$ using ACM.
			\STATE Calculate ${d}^0_{i}(u_i^r)= (G+\nabla^2 J(u_i^r))^{-1}\nabla J(u_i^r)$.
			\FOR {$l=1,...,r$}
			\STATE $d_i^{l}(u_i^r)= (G+\nabla^2 J(u_i^r))^{-1}\big(\nabla J(u_i^r)+Gd_i^{l-1}(u_i^r)\big)$
			\ENDFOR	
			\STATE $u_i^{r+1} = u_i^r-d_i^r(u_i^r)$.
			\STATE If \( ||u_i^{r+1} - u_i^r||<\epsilon \) for each agent $i$, execute Step 14; otherwise, set $r=r+1$ and return to Step 5.
			\STATE Extract $u_i(t_k|t_k)$ from $u_i^{r+1}$ and obtain $x_i(t_{k+1})$.
			\STATE	Set $t_k=t_{k+1}$ and return to step 3.
			\ENDFOR	
		\end{algorithmic}
	\end{algorithm}
	
	\subsection{Finite-time leader-follower optimal consensus} \label{555}
	Different from  the leaderless case, in the presence of a leader, the states of the agents not only need to achieve consensus but also must track the state of the leader \( x_l(k) \). To achieve the finite-time leader-follower optimal consensus control objective, a necessary assumption is imposed.
	\begin{assumption}\label{asmpt2}
		The directed graph $\mathcal{G}$ has a spanning tree.
	\end{assumption}
	
	Similar to the leaderless case, we also define a cost function here to achieve finite-time leader-follower optimal consensus. We consider a finite control horizon $N$ and define the control period from $k=0$ to $k=N$. Consider a leader-follower multi-agent system consisting of \( n \) followers and one leader, where the leader moves along a predefined trajectory, and among the \( n \) followers, some agents can communicate directly with the leader. For ease of representation, when  $i  \in \{1, \ldots, m-1\}$ , the \( i \)-th agent does not communicate directly with the leader; when $ i  \in  \{m, \ldots, n\}$, it can communicate directly with the leader. The global cost function is as follows:
	\begin{equation}
		\begin{aligned}	\label{costf2}
			J_2 = {} & \sum_{k = 0}^N \Bigg( \sum_{i = 1}^{m-1} \bigg( \sum_{j \in \mathcal{N}_i} (x_i(k) - x_j(k))^{T} Q_{ij} (x_i(k) - x_j(k)) \\
			& + u_i(k)^{T} R_i u_i(k) \bigg) +\ \sum_{i = m}^{n} \bigg( \sum_{j \in \mathcal{N}_i} (x_i(k) - x_j(k))^{T} Q_{ij}  \\
			& \times(x_i(k)- x_j(k)) + (x_i(k) - x_l(k))^{T} W_{il} (x_i(k) - x_l(k)) \\
			& +\ u_i(k)^{T} R_i u_i(k)\  \bigg)\ \Bigg)\  +\ \ \sum_{i = 1}^{n}\  \bigg(\ \sum_{j \in \mathcal{N}_i}\ (\ x_i(N+1)  \\
			& -\ x_j(N+1))^{T}D_{ij}(\ x_i(N+1)\ -\ x_j(N+1)\ )\ \bigg)  \\
			&+\ \sum_{i = m}^{n} \bigg(\ (x_i(N+1) - x_l(N+1))^{T} E_{il} (x_i(N+1)\\
			& - x_l(N+1)) \bigg),
		\end{aligned}
	\end{equation}
	where $Q_{ij}\geq0$, $W_{il}\geq0$, $D_{ij}\geq0$, $E_{il}\geq0$ and $R_{i}>0$ are weighted matrices. 
	
	For the leader-follower case, the optimal consensus control problem is designing a distributed protocol that derives the optimal control input sequence $u_i$ for each agent to minimize the (\ref{costf2}) subject to (\ref{nonagents}), which can be formulated as follows.
	
	\textbf{Problem 2}: The finite-time optimal consensus control problem for the leader-follower case:
	\begin{equation}
		\begin{array}{l}
			\mathop {\rm minimize}\limits_{U} J_2. \\
			\mathop{s.t.\ (\ref{nonagents})}
		\end{array}  \label{minJ2}
	\end{equation}
	
	Following the same methodology employed in Problem 1, we apply Pontryagin's Maximum Principle to the system described in (\ref{nonagents}) with the cost functional (\ref{costf2}). This yields the following costate equations and equilibrium conditions for optimality:
	\begin{equation}
		\begin{array}{l}
			0= R_iu_i(k) + \lambda_i(k+1)\frac{\partial f_i(x_{i}(k),u_{i}(k))}{\partial u_{i}(k)},\label{ph2}
		\end{array}
	\end{equation}
	\begin{equation} 
		\lambda_i(k) = 
		\begin{cases}
			\begin{aligned}
				&\sum\limits_{j \in \mathcal N_i} Q_{ij}(x_i(k)-x_j(k)) \\
				&+\lambda_i(k+1)\frac{\partial f_i(x_{i}(k),u_{i}(k))}{\partial x_{i}(k)}
			\end{aligned},  i \in \{1,\dots,m-1\} \\
			\begin{aligned}
				&W_{il}(x_i(k)-x_l(k)) \\
				&+\sum\limits_{j \in \mathcal N_i} Q_{ij}(x_i(k)-x_j(k)) \\
				&+\lambda_i(k+1)\frac{\partial f_i(x_{i}(k),u_{i}(k))}{\partial x_{i}(k)}
			\end{aligned},  \ i \in \{m,\dots,n\} \\ 
		\end{cases}
		\label{costate2} 
	\end{equation}
	with the terminal value ${\lambda_i(N+1)} = \sum_{j \in\mathcal N_i}D_{ij}(x_i(N+1)-x_j(N+1))$ when $i \in \{1,2,\dots,m-1\}$ and ${\lambda_i(N+1)} = \sum_{j \in \mathcal N_i}D_{ij}(x_i(N+1)-x_j(N+1))+E_{il}(x_i(k)-x_l(k))$ when $i \in \{m,\dots,n\}$.
	
	It is noteworthy that the solution to (\ref{ph2})-(\ref{costate2}) can still be computed in a distributed manner. Despite the presence of a leader, from the perspective of each agent, the aforementioned equations only involve the state information of itself and its neighbors. The  MSA mentioned in Problem 1 can also be applied to solve (\ref{ph2})-(\ref{costate2}). However, its relatively slow computational speed makes it challenging to deploy online in practical applications. We aim to use  (\ref{cs_ocp}) for the leader-follower case, as this would significantly improve computational efficiency. Reviewing (\ref{cs_ocp}), its applicability hinges on whether \(\nabla J(u_i)\) and \(\nabla^2 J(u_i)\) can be computed using only neighbors information. The computation of \(\nabla J(u_i)\) and \(\nabla^2 J(u_i)\) relies on ACM. We observe that in the leader-follower case, compared to the leaderless scenario, the \(n-m\) agents directly communicating with the leader incorporate an additional term \(W_{il}(x_i^r(k)-x_l^r(k))\) to the original expression \(\sum_{j \in \mathcal N_i}Q_{ij}(x_i^r(k)-x_j^r(k))\) when using ACM. However, this modification does not compromise the distributed nature of the computation, since the leader is a neighbor of these agents. Notably, the leader-follower case preserves the algorithmic structure of ACM without modification, demonstrating the unified nature of our approach across different multi-agent system configurations. 
	
	For the leader-follower case, we directly present the MPC-based algorithm here, whose benefits have been previously established. For a comprehensive implementation, please refer to Algorithm 3. 
	
	\begin{algorithm}[h]
		\renewcommand{\algorithmicrequire}{\textbf{Input:}}
		\renewcommand{\algorithmicensure}{\textbf{Output:}}
		\caption{MPC-based leader-follower optimal consensus algorithm for nonlinear multi-agent systems}
		\label{alg3}
		\begin{algorithmic}[1]
			\STATE Initialization: Each agent checks whether it can communicate with its neighbors,  sets $x_i(t_0) \in {\mathbb{R}^p}$ and initialize the control sequence $u_i^0=[u_i^0(t_0)^T,\dots,u_i^0(t_0+N_p)^T ]^T$. Set $\epsilon>0$.		
			\FOR {each agent $i=1,...,n$}
			\WHILE{$t=t_k,k=0,1,2,...$}
			\item Set $r=0$ and  $x_i^0(t_k|t_k)=x_i(t_k)$. 		
			\ENDWHILE
			\STATE  Calculate  $\{ x_i^r(t_k|t_k),\dots,x_i^r(t_k+N_p|t_k)\}$ by $u_i^r$ from (\ref{nonagents}).
			\STATE Send $\{ x_i^r(t_k|t_k),\dots,x_i^r(t_k+N_p|t_k)\}$ to its neighbors and receive $\{ x_j^r(t_k|t_k),\dots,x_j^r(t_k+N_p|t_k)\}$ from its neighbor $j \in \mathcal N_i$. For agents that can communicate directly with the leader ($i \in{m,...,n}$), they also need to receive $\{ x_l^r(t_k|t_k),\dots,x_l^r(t_k+N_p|t_k)\}$ from the leader.
			\STATE Calculate $\nabla J(u_i^r)$ and $\nabla^2 J(u_i^r)$ using ACM.
			\STATE Calculate ${d}^0_{i}(u_i^r)= (G+\nabla^2 J(u_i^r))^{-1}\nabla J(u_i^r)$.
			\FOR {$l=1,...,r$}
			\STATE $d_i^{l}(u_i^r)= (G+\nabla^2 J(u_i^r))^{-1}\big(\nabla J(u_i^r)+Gd_i^{l-1}(u_i^r)\big)$
			\ENDFOR	
			\STATE $u_i^{r+1} = u_i^r-d_i^r(u_i^r)$.
			\STATE If \( ||u_i^{r+1} - u_i^r||<\epsilon \) for each agent $i$, execute Step 14; otherwise, set $t=t+1$ and return to Step 3.
			\STATE Extract $u_i(t_k|t_k)$ from $u_i^{r+1}$ and obtain $x_i(t_{k+1})$.
			\STATE	Set $t_k=t_{k+1}$ and return to step 3.
			\ENDFOR	
		\end{algorithmic}
	\end{algorithm}
	\subsection{Technical notes of the proposed algorithms} 
	In this subsection, we discuss the compelling characteristics of our proposed algorithms and some simplification techniques in practical use, which are summarized as follows:
	\begin{itemize}	
		\item
		In leaderless multi-agent systems, the configuration of weight matrices $Q_{ij}$ offers two feasible schemes: one may either assign independent weight matrices to each nonlinear agent or adopt a unified inter-neighbor weight matrix. The latter demonstrates greater simplicity in engineering implementation. For leader-follower architectures, setting $W_{il}=Q_{ij}$ reduces the cost function $J_2$ to the form of $J_1$. The explicit separation of leader-connected weights $W_{il}$ from ordinary neighbor-connected weights $Q_{ij}$ in practical design serves to distinguish their respective roles. Simulation experiments reveal that when employing diagonal matrix forms $W_{il}=aI$ and $Q_{ij}=bI$ (where $a\geq 0$, $b\geq 0$), although the algorithm maintains convergence at $a=b$, the system's dynamic performance shows significant improvement when $a>b$.
		\item
		Although this study focuses on nonlinear multi-agent systems, the proposed algorithm remains applicable to linear multi-agent systems as well. In fact, linear systems can be regarded as a special case of nonlinear systems, since—whether the system is nonlinear or linear—applying  Pontryagin's maximum principle to the multi-agent system yields costate equations and equilibrium conditions of the same form . Numerous studies have investigated distributed optimal consensus control for linear systems \cite{bib11},\cite{bib29}, \cite{bib30},\cite{bib38} and the algorithms proposed in this paper can also be validated with such linear models. 
		\item
		The MPC strategy not only enables online deployment of our algorithm but also provides strong disturbance rejection capabilities. Through its receding horizon optimization and feedback correction mechanisms, MPC strategy can handle both internal and external dynamic disturbances in real time. For sudden collision avoidance requirements between agents during operation, MPC quickly responds by replanning trajectories online. When communication delays cause abnormal neighbor state information, its prediction-model-based robustness maintains formation stability. This disturbance rejection capability allows the multi-agent system to achieve mission objectives even under disturbances, fully demonstrating the practical engineering value of the MPC strategy in optimal consensus control.
	\end{itemize}
	
	\section{CONVERGENCE ANALYSIS}           
	In this section, we conduct a convergence analysis of the proposed algorithms. The results show that the superlinear convergence rate of the algorithm demonstrates the superiority of our optimal control-based design. Additionally, the convergence analysis of MPC-based derivative algorithms further verifies the applicability of our approach.
	
	For convenience of future discussion, some symbols are denoted below. In fact, our convergence analysis is conducted from a global perspective, where the preceding distributed algorithm is represented in a centralized form to facilitate the analysis process.
	\begin{equation*}
		\begin{aligned}
			x(k) &= [x_{1}^{T}(k),\ldots,x_{n}^{T}(k)]^{T}, \\
			U &= [u_{1}^{T},\ldots,u_{n}^{T}]^{T}, \\
			Q &= \operatorname{diag}\{Q_{1j},\ldots,Q_{nj}\}, \\
			R &= \operatorname{diag}\{R_{1},\ldots,R_{n}\}, \\
			G &= \operatorname{diag}\{G_{1},\ldots,G_{n}\}, \\
			g(U) &= [\nabla J(u_{1})^{T},\ldots,\nabla J(u_{n})^{T}]^{T}, \\
			H(U) &= \operatorname{diag}\{\nabla^{2}J(u_{1}),\ldots,\nabla^{2}J(u_{n})\}.
		\end{aligned}
	\end{equation*}
	
	\begin{assumption}\label{assm2}
		There exist constants $0<m_1\le m_2<\infty$ such that for any $x \in {\mathbb{R}^m}$, $m_1I\preceq \nabla^2J(u_i)\preceq m_2I$.
	\end{assumption}
	\begin{remark}
		Assumption \ref{assm2} is standard in the convergence analysis, see \cite{bib39}. In practice, for controllable nonlinear systems, the conditions of Assumption \ref{assm2} are not easily satisfied because they involve the convexity of the cost function $J$ with respect to the control sequence. However, it should be noted that if we regard any non-strongly convex $ J(U) $ as a strongly convex function within a neighborhood  $\mathcal{N}$  of some local minimum point, our algorithm will converge to the local minimum when the initial guess  $U^0 \in \mathcal{N}$  and the scaling matrix  $G$ is chosen sufficiently large (i.e., greater than a certain lower bound).
	\end{remark}
	
	\begin{theorem}\label{thmmain}
		The sequence$\{U^r\}$ generated by Algorithm 1 converges to the optimal solution of Problem 1, and there exists a scalar $c_1 > 0$ and an integer $r_0 > 0$ such that for any $r > r_0$, 
		\begin{align}
			\|U^{r+1}-U^*\|\le c_1 \rho((G+H(U^r))^{-1}G)^{r+1} \|U^r-U^*\|. \label{conrate3}
		\end{align}
		That is to say,  Algorithm 1 is superlinearly convergent.
	\end{theorem}
	\begin{proof}
		Let $\phi(U)=[I-((G+H(U))^{-1}G)^{r+1}]H(U)^{-1}$ and 	Algorithm 1 can be written in centralized form as 
		\begin{equation}
			U^{r+1}=U^r-\phi(U^r)g(U^r). \label{ocpsp}
		\end{equation}
		
		Then we will show that $\phi(U)$ is symmetric as follows.	
		\begin{align}
			&\bar \phi(U)\notag\\
			=&((I+G^{-1}H(U))^{-1})^{r+1}H(U)^{-1} \notag\\
			=&(I+G^{-1}H(U))^{-1}(I+G^{-1}H(U))^{-1}\dotsm\notag\\
			&\times (I+G^{-1}H(U))^{-1}H(U)^{-1}\notag\\
			=&[H(U)^{-1}H(U)(I+G^{-1}H(U))^{-1}]\notag\\
			&\times [H(U)^{-1}H(U)(I+G^{-1}H(U))^{-1}]\dotsm \notag\\
			&\times [H(U)^{-1}H(U)(I+G^{-1}H(U))^{-1}]H(U)^{-1} \notag\\
			=&[H(U)^{-1}(I+G^{-1}H(U))^{-1}H(U)]\notag\\
			&\times [H(U)^{-1}(I+G^{-1}H(U))^{-1}H(U)]\dotsm \notag\\
			&\times [H(U)^{-1}(I+G^{-1}H(U))^{-1}H(U)]H(U)^{-1} \notag\\
			=&H(U)^{-1}(I+H(U)G^{-1})^{-1}(I+H(U)G^{-1})^{-1}\dotsm \notag\\
			&\times (I+H(U)G^{-1})^{-1} \notag\\
			=&H(U)^{-1}G(G+H(U))^{-1}G(G+H(U))^{-1}\dotsm \notag\\
			&\times G(G+H(U))^{-1} \notag\\
			=&H(U)^{-1}[G(G+H(U))^{-1}]^{r+1}=\bar \phi(U)^T.
		\end{align} 
		Since $\phi(U)=H(U)^{-1}-\bar \phi(U)$, it shows that $\phi(U)=\phi(U)^T$.
		
		For the \( (G+H(U))^{-1}G \) part in \( \phi(U) \), according to Corollary 4.6.3 in \cite{bibn2}, its eigenvalues are positive. Let $\psi > 0$ and $U$ be  any eigenvalue of $(G+H(U))^{-1}G$ and its corresponding eigenvector, respectively. Then $(1-\psi)U^TGU=\psi U^TH(U)U>0$, which indicates $1-\psi>0$. Accordingly, $0<\psi<1$ and $\rho ((G+H(U)^{-1})G)<1$. Let $H(U)^{-1}=Y^TY$ since $H(U)>0$, then $\phi(U)$ can be rewritten as $Y^{-1}Y(I-((G+H(U)^{-1}G))^{r+1})Y^{T}Y$.  It is easy to know  that $\phi(U)$ and $Y(I-((G+H(U)^{-1}G))^{r+1})Y^{T}$ are similar and $Y(I-((G+H(U)^{-1}G))^{r+1})Y^{T}$ and $I-((G+H(U))^{-1}G)^{r+1}$ are congruent. Recall that each eigenvalue $\psi$ of $(G+H(U))^{-1}G$ admits $0<\psi<1$, so is each one of $((G+H(U))^{-1}G)^{r+1}$. Further, any eigenvalue of $I-((G+H(U))^{-1}G)^{r+1}$ is positive and
		less than one. Consequently, $\phi(U)>0$.
		
		Let $v_t(U)=\phi(U)g(U)$. Following from the direct derivation, it shows
		\begin{align}
			v_t(U^*)=&[I-((G+H(U^*))^{-1}G)^{r+1}]{H(U^*)}^{-1}g(U^*)\nonumber\\
			=&0, \label{hatgkxstar}\\
			v'_t(U^*)=&I-((G+H(U^*))^{-1}G)^{r+1}. \label{dehatgkxstar}
		\end{align}
		From \eqref{ocpsp}, it follows that
		\begin{align}
			&U^{r+1}-U^*\notag\\
			=&U^r-U^*- v_t(U^r)\notag\\
			\approx&U^r-U^*-[ v_t(U^*)+ v'_t(U^*)(U^r-U^*)]\notag\\
			=&U^r-U^*- v'_t(U^*)(U^r-U^*)\notag\\
			=&((G+H(U^*))^{-1}G)^{r+1}(U^r-U^*).
		\end{align} 
		In the third line, the Taylor expansion of \({v}_t(U)\) at \(U^*\) has been utilized. Given \(G > 0\), we have \(G = F^T F\), and the following equalities hold:  
		\[
		\begin{aligned}
			(G +  H(U^*))^{-1} G 
			&= (G +  H(U^*))^{-1} F^T F \\
			&= F^{-1} F (G + H(U^*))^{-1} F^T F \\
		\end{aligned}
		\] 
		The second equality implies:  
		\[
		(G + H(U^*))^{-1} G = \overline{F}^{-1} \, \overline{\Lambda} \, \overline{F},
		\]  
		where \(\overline{\Lambda}\) is a diagonal matrix. From the preceding discussion and $\overline{F}$ is nonsingular, it is immediate to get  $\rho(\overline{\Lambda})=\rho((G+H(U^*))^{-1}G)<1$ and
		\[
		\begin{aligned}
			&\ \ \ \ \|U^{r+1} - U^*\| \\
			&=\left\| \overline F^{-1} \overline \Lambda^{r+1} \overline F (U^r - U^*)  \right\| \\
			&\leq \rho((G+H(U^*))^{-1}G)^{r+1}\|\overline F^{-1}\| \|\overline F\| \|U^r-U^*\|
		\end{aligned}
		\]
		
		Let $c_1 = \| \overline F^{-1}\| \|\overline F\|$. Now it is evident that (\ref{conrate3}) holds. The proof is completed.
	\end{proof}

	For the MPC-based control algorithm, we have previously analyzed that the presence of a leader does not affect our approach, so here we will conduct a convergence analysis of Algorithm 2. Although section III presents an MPC-based algorithm and explains the reasoning behind it, we will rigorously reformulate the consensus problem for nonlinear multi-agent systems within the MPC framework here for the sake of proof completeness.
	
	Similarly to (\ref{costf}), the cost function at predictive time instant $t_k$ is set to be
	
	\begin{align}\label{mpccostf}
		J_{t_k}(U_{\tau|t_k})\!=\!&\sum\limits_{\tau = t_k}^{t_k+N_P-1} (\sum\limits_{i = 1}^n(\sum\limits_{j \in \mathcal N_{i}}(x_i(\tau|t_k)-x_j(\tau|t_k))^{T}Q_{ij}\nonumber\\
		&+\!\sum\limits_{i = 1}^n(\sum\limits_{j \in \mathcal N_{i}}(x_i(t_k+N_p|t_k)-x_j(t_k+N_p|t_k))^{T}\nonumber\\
		&\!\times \!D_{ij}(x_i(t_k+N_p|t_k)-x_j(t_k+N_p|t_k))),
	\end{align}

	At this point, each nonlinear multi-agent needs to employ Algorithm 2 to minimize $ J_{t_k}(U_{\tau|t_k})$ at time step $t_k$. 
	
	\begin{remark}
		Within the MPC framework, Algorithm 2 is designed to incorporate real-time state feedback. This architecture enables dynamic recalibration of the control input sequence during operation, thereby providing adaptive stabilization for the closed-loop control system. The online adjustability characteristic effectively compensates for system uncertainties and disturbances through continuous optimization of control actions.
	\end{remark}
	
	We now proceed to conduct a convergence analysis of the MPC-based Algorithm 2.
	\begin{theorem}\label{thmmain2}
		For the nonlinear multi-agent systems (\ref{nonagents}), the consensus can be achieved asymptotically with the control input generated by Algorithm 2 under Assumption \ref{asmpt1}, i.e,
		\begin{align}
			\lim_{k \to \infty} (x_i(k)-x_j(k))=0. \notag
		\end{align}
	\end{theorem}
	\begin{proof}
		Assume that $u_i^*(\tau|t_k)$ is the optimal solutions at the predictive time $t_k$ with $\tau \in [t_k,t_k+N_p]$ for each agent $i$. We can obtain the feasible solution at time \( t_{k+1} \) as follows:  
		
		\begin{align}
			{u_i}(\tau|t_{k+1}) =
			\begin{cases} 
				{u_i}^*(\tau|t_k), &  \tau \in [t_{k+1}, t_k+N_p-1], \\
				{u}_{i}(\tau|t_{k+1}), &  \tau =t_{k+1}+N_p-1 .
			\end{cases} \label{fs22}
		\end{align} 
		
		The feasible system state can be obtained from the above controller and is denoted as:  
		
		\begin{align}
			{x_i}(\tau|t_{k+1}) = \begin{cases} 
				{x}^*(\tau|t_k)  & \ \tau \in [t_{k+1}, t_{k}+N_p] \\
				{x}_{i}(\tau|t_{k+1}) &  \tau = t_{k+1}+N_p.
			\end{cases} \label{fs23}
		\end{align} 
		For the optimal control input $u_i^*(\tau|t_{k+1})$ at time $t_{k+1}$, the following inequality holds:	
		\begin{align}
			&J_{t_{k+1}}(U^*_{\tau|t_{k+1}}) - 	J_{t_{k}}(U^*_{\tau|t_{k}}) \leq J_{t_{k+1}}(U_{\tau|t_{k+1}})-J_{t_{k}}(U^*_{\tau|t_{k}}) \notag \\
			=&\sum\limits_{i = 1}^{n} \bigg \{ \sum\limits_{\tau = t_{k+1}}^{t_{k+1}+N_p-1} \big( (x_i(\tau|t_{k+1})-x_j(\tau|t_{k+1}))^TQ_{ij} \notag \\
			&\times(x_i(\tau|t_{k+1})-x_j(\tau|t_{k+1}))+ u_i(\tau|t_{k+1})^TR_{i} u_i(\tau|t_{k+1}) \big) \notag \\
			&+(x_i(t_{k+1}+N_p|t_{k+1})-x_j(t_{k+1}+N_p|t_{k+1}))^TD_{ij} \notag \\
			&\times(x_i(t_{k+1}+N_p|t_{k+1})-x_j(t_{k+1}+N_p|t_{k+1})) \notag \\
			&-\sum\limits_{\tau = t_{k}}^{t_{k}+N_p-1}\big( (x_i^*(\tau|t_{k})-x_j^*(\tau|t_{k}))^TQ_{ij}(x_i^*(\tau|t_{k}) \notag  \\
			&-x_j^*(\tau|t_{k}))+ \notag u_i^*(\tau|t_{k})^T R_{i} u_i^*(\tau|t_{k}) \big)- (x_i^*(t_{k}+N_p|t_{k}) \notag \\ 
			&-x_j^*(t_{k}+N_p|t_{k}))^TD_{ij}(x_i^*(t_{k}+N_p|t_{k}) \notag \\
			&-x_j^*(t_{k}+N_p|t_{k})) \bigg \} \label{inqJ1}.
		\end{align} 
		In ACM, similar to \cite{bib28}, we set \( D_{ij} = 0 \) for computational convenience, in which case \( u_i(t_{k+1} + N_p - 1|t_{k+1}) \) can also be zero.  Substituting the feasible solution (\ref{fs22})-(\ref{fs23}) into (\ref{inqJ1}) and obviously we have
		\begin{align}
			&J_{t_{k+1}}(U^*_{\tau|t_{k+1}}) - 	J_{t_{k}}(U^*_{\tau|t_{k}}) \leq \notag\\
			&- \sum\limits_{i = 1}^{n} \bigg \{ (x_i^*(t_k|t_{k})-x_j^*(t_k|t_{k}))^TQ_{ij}(x_i^*(t_k|t_{k})-x_j^*(t_k|t_{k}))\notag \\
			&- u_i^*(t_k|t_{k})^TR_{i}u_i^*(t_k|t_{k})\bigg \} \label{inqJ2},
		\end{align} 
		which implies that $\{ J_{t_k}(U^*_{\tau|t_k})\}$ is nonincreasing. Because the quadratic cost function is non-negative, we can conclude that $\{ J_{t_k}(U^*_{\tau|t_k})\}$ converges as \( t_k \) approaches infinity. In other words, there exists a constant $\zeta \geq$ such that
		\begin{align}
			\textstyle \lim_{t_k \to \infty} J_{t_k}(U^*_{\tau|t_k})=\zeta. \notag
		\end{align}
		It should be noted that the consensus error $x^*_i(\tau|t_k)-x^*_j(\tau|t_k)$ and the control input $u^*_i(\tau|t_k)$ are bounded for all $t_k$. As established in \cite{bib40}, there exists a non-empty, compact, and positively invariant set \( \Omega \) that contains all positive limit points of the vector sequence $[(x_i^*(\tau|t_k)-x_j^*(\tau|t_k))^T,u_i^*(\tau|t_k)^{T}]^T$ as $t_k \to \infty$. Consequently, for any element \( \omega_i \in \Omega \), there exists a subsequence such that:
		\[
		[(x_i^*(\tau|t_{k}) - x_j^*(\tau|t_{k}))^T, u_i^*(\tau|t_{k}^T)]^T \to \omega_i ,
		\]
		with \( J(\omega_i) = \zeta \). The positive invariance of \( \omega_i \) implies that
		
		\begin{align}
			\Omega\subset\left\{x_{i}^{\star}(\tau|t_{k}),u_{i}^{\star}(\tau|t_{k})\mid J_{t_{k+1}}(U^*_{\tau|t_{k+1}})-J_{t_k}(U^*_{\tau|t_k})=0\right\}.
		\end{align}
		Hence, (\ref{inqJ2}) implies that
		\begin{align}
			\Omega\subset \{&x_{i}^{\star}(\tau|t_{k}),u_{i}^{\star}(\tau|t_{k})\mid \| x_i^*(t_k|t_k)-x_j^*(t_k|t_k)\|^2=0,\notag \\  
			&\|u_i^*(t_k|t_k)\|^2=0, \quad i\in \{1,2,...,n\}, j \in \mathcal{N}_i \}. 
		\end{align}
		Then we can show that $\textstyle \lim_{t_k \to \infty}\| x_i^*(\tau|t_k)-x_j^*(\tau|t_k)\|= 0.$ From Theorem 2,
		it follows that consensus can be achieved asymptotically by using Algorithm 2. The proof is completed.
	\end{proof}
	
	\begin{remark}
		As previously mentioned, the optimal consensus control for nonlinear multi-agent systems within the optimal control theory framework can uniformly address both leader-following and leaderless scenarios. Consequently, the convergence analysis process for Algorithm 3 is identical to that of Algorithm 2, with the assumption conditions further relaxed to Assumption \ref{asmpt2}. In other words, under Assumption \ref{asmpt2}, the nonlinear multi-agent system can achieve leader-follower optimal consensus through Algorithm 3.
	\end{remark}

	\section{SIMULATION EXAMPLES}
	To further demonstrate the superiority of the algorithms, this section applies the proposed distributed optimal consensus control algorithms to solve several practical problems. Specifically, for the leaderless case, we test our Algorithm 2 on the AGVs rendezvous task; for the leader-follower case, we adopt the multi-agent system from \cite{bib30} and compare the results of our Algorithm 3 with those of the algorithm proposed in \cite{bib30}; additionally, we demonstrate the performance of our algorithm in formation control problems. All our experiments are executed on the same computer with a 3.2 GHz Intel Core $i$-5 processor with Matlab R2020a.
	\subsection{Consensus  problem for leaderless case} 
	Consider the following multi-agent system consisting of four AGVs, the  kinematic model of each AGV is
	\begin{equation}
		\begin{cases}
			\dot x_i(t) = v_i(t)\cdot{\rm cos}\theta_i(t), \\
			\dot y_i(t) = v_i(t)\cdot{\rm sin}\theta_i(t),\\
			\dot \theta_i(t) = \omega_i(t),
		\end{cases}                         \label{ODE_AGV}
	\end{equation}
	where $t\in[0,t_f]$. $x_i$ and $y_i$ are the position for each agent in the plane, while $\theta_i$ represents the AGV's orientation. $v$ and $\omega$ are the control inputs.  The discretized form of (\ref{ODE_AGV}) can be expressed as:
	\begin{equation}
		\begin{cases}
			x_i(k+1) = x_i(k) + \Delta\cdot v_i(k){\rm cos}\theta_i(k) , \\
			y_{i}(k+1) = y_i(k) + \Delta\cdot v_i(k){\rm sin}\theta_i(k),\\
			\theta_{i}(k+1) = \theta_i(k) + \Delta\cdot \omega_i(k),
		\end{cases}                         \label{AGV}
	\end{equation}
	where $k=0,1,...,N-1$ represents the sampling times, and $\Delta$ is the time step.
	
	The aforementioned AGV model has been widely applied in various domains, including warehouse logistics, industrial automation, and intelligent transportation systems. Notably, numerous studies such as \cite{bib41, bib42, bib43}, have adopted this model for algorithm validation and performance analysis. Here we employ this model to verify the effectiveness of our leaderless consensus control algorithm. The communication topology for four AGVs is shown in Fig. \ref{ctopology}.
	\begin{figure}[htbp]
		\captionsetup{font={footnotesize}}
		\centering
		\includegraphics[width=0.8\linewidth]{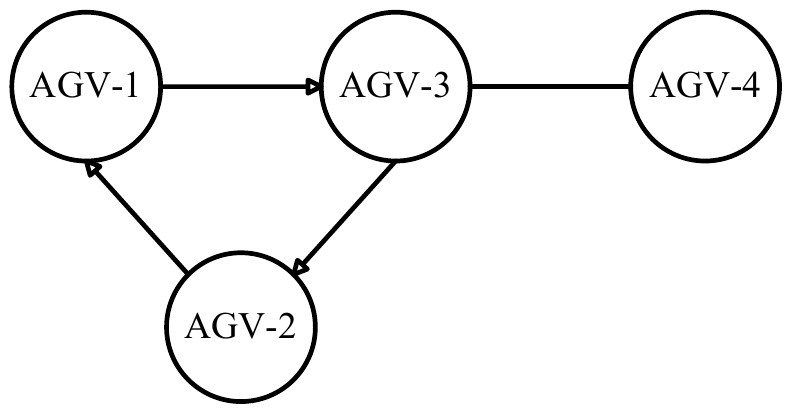}
		\caption{ Communication topology of AGVs}
		\label{ctopology}
	\end{figure}
	
	In this example, We employ Algorithm 2 to accomplish the rendezvous task for four AGVs. Specifically, four AGVs are positioned at different locations in the room, where each AGV can only communicate with its neighboring AGVs. The objective is for each AGV to rapidly converge to the same position using its own optimal consensus control protocol. The time step $\Delta=0.05$ and the prediction horizon $N_p=8$. The initial pose for each AGV is set as $\{x_1(0)=0.10,y_1(0)=0.30,\theta_1(0)=0.78\}$, $\{x_2(0)=6.20,y_2(0)=0.10,\theta_2(0)=2.36\}$, $\{x_3(0)=6.00,y_3(0)=6.00,\theta_3(0)=3.93\}$, $\{x_4(0)=-0.1,y_4(0)=6.4,\theta_4(0)=-0.78\}$ and the weight matrix $Q_{ij}={\text{diag}}[10,10,10], R_i={\text{diag}}[0.1,0.1],\ i\in\{1,2,3,4\}$.
	
	Fig. \ref{ctraj} shows the motion trajectories of the AGVs, illustrating smooth convergence to the optimal consensus point. Here we define the consensus error in our study as the Euclidean distance between neighboring agents. The consensus error between each AGV and its neighbors is shown in Fig. \ref{cerror}, which demonstrates rapid and stable convergence to zero. The simulation results demonstrate that Algorithm 2 can effectively solve the leaderless consensus problem in nonlinear multi-agent systems.
	\begin{figure}[htbp]
		\captionsetup{font={footnotesize}}
		\centering
		\includegraphics[width=0.85\linewidth]{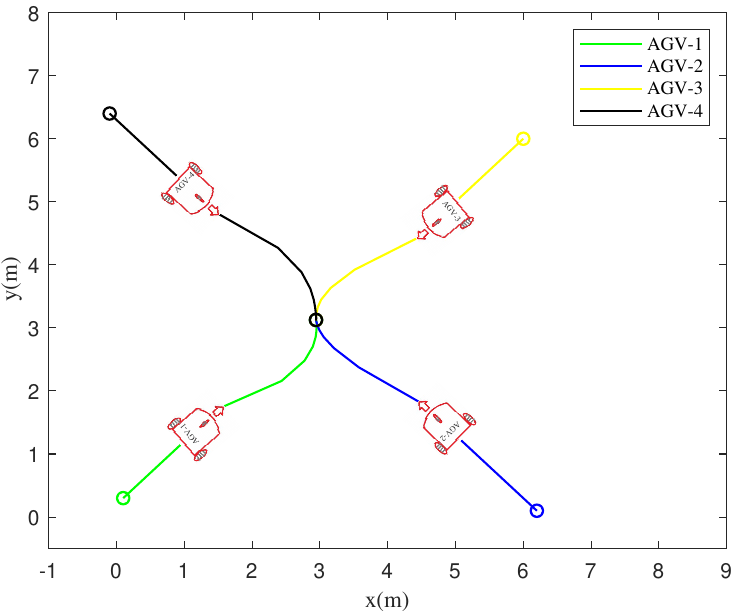}
		\caption{ Actual trajectories of AGVs}
		\label{ctraj}
	\end{figure}

	\begin{figure}[htbp]
		\captionsetup{font={footnotesize}}
		\centering
		\includegraphics[width=0.85\linewidth]{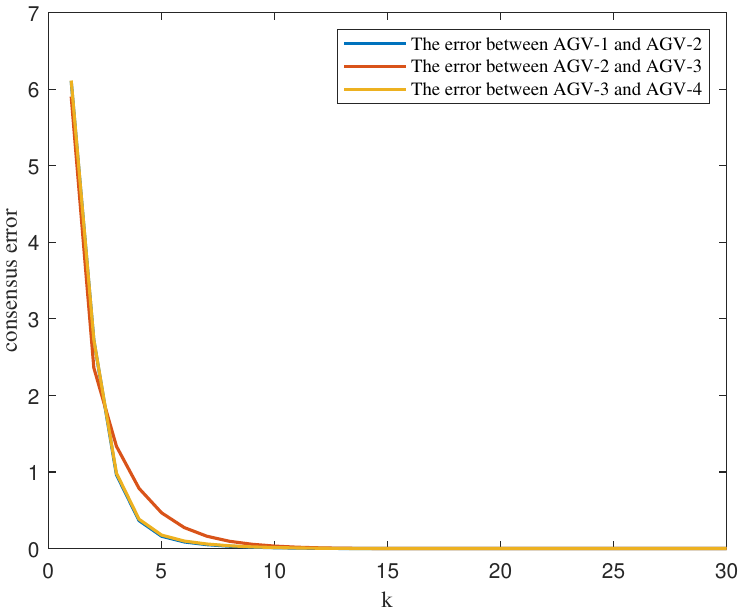}
		\caption{ Consensus errors of AGVs}
		\label{cerror}
	\end{figure}
	
	\subsection{Consensus  problem for leader–follower case} 
	Consider the following multi-agent system consisting of four followers and one leader. The system model for each follower is given by 
	\begin{equation}
		x_i(k+1) = Ax_i(k) + B(u_i(k)+f(x_i(k))), \label{vsagentf}
	\end{equation}
	with $$ A=\left[\begin{array}{cc}
		0.898 & 0.056\\
		0.968 & -0.084
	\end{array}\right],\quad B=\left[\begin{array}{l}
		0.87\\
		-1.8
	\end{array}\right] $$
	and $$ f(x_i(k))=\left[\begin{array}{cc}
		0.01sin(x_{i1}(k)) \\
		0.01sin(x_{i2}(k))
	\end{array}\right].$$
	
	The dynamic of leader is given by
	\begin{equation}
		x_l(k+1) = A_lx_l(k) + B_l(f_l(x_l(k))+h_l(k)), \label{vsagentl}
	\end{equation}
	with $$ A_l=\left[\begin{array}{cc}
		0.898 & 0.056\\
		0.968 & -0.084
	\end{array}\right],\quad B_l=\left[\begin{array}{l}
		0.87\\
		-1.8
	\end{array}\right] $$
	and $$ f_l(x_i(k))=\left[\begin{array}{cc}
		0.01sin(x_{i1}(k)) \\
		0.01sin(x_{i2}(k))
	\end{array}\right], h_l(k)=\begin{array}{l}0.1sin(0.05k).
	\end{array} $$
	
	The communication topology for four followers and one leader is shown in Fig. \ref{ctopology2}.
	\begin{figure}[htbp]
		\captionsetup{font={footnotesize}}
		\centering
		\includegraphics[width=0.7\linewidth]{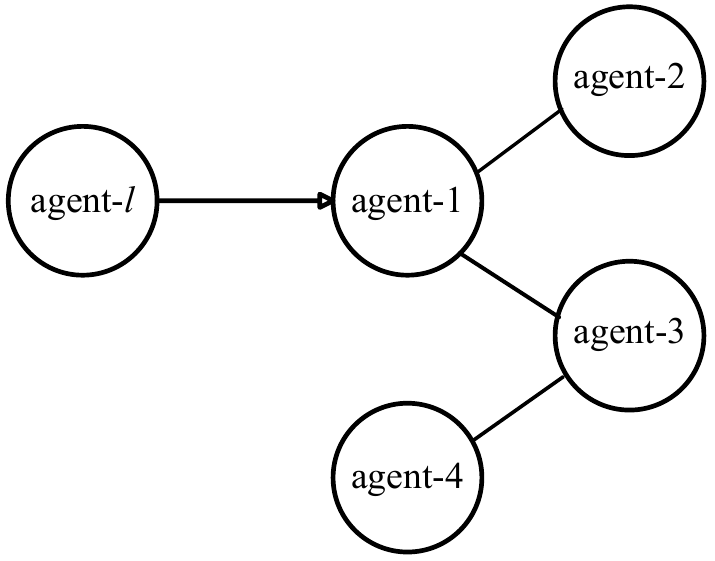}
		\caption{ Communication topology of agents}
		\label{ctopology2}
	\end{figure}
	
	In this example, we employ Algorithm 3 to address the leader-follower consensus control problem. For notational convenience, we designate the four follower agent as agent-1 through agent-4, and denote the leader as agent-$l$. For comparative purposes, the algorithm proposed in \cite{bib30} is also implemented for this problem. The initial states of the multi-agent system are chosen as follows:
	$$ 
	x_{l}(0)=\begin{bmatrix}
		-8.14 \\ 
		30.33
	\end{bmatrix}, 
	x_{1}(0)=\begin{bmatrix}
		-12.23 \\ 
		8.93
	\end{bmatrix}, 
	x_{2}(0)=\begin{bmatrix}
		-14.31 \\ 
		2.08
	\end{bmatrix} 
	$$
	
	$$
	x_{3}(0)=\begin{bmatrix}
		-4.11 \\ 
		-1.31
	\end{bmatrix}, 
	x_{4}(0)=\begin{bmatrix}
		-4.57 \\ 
		14.28
	\end{bmatrix}.
	$$
	
	Regarding Algorithm 3, we configure the parameters with prediction horizon $N_p=8, Q_{ij}=30, W_{1l}=80$, and $R_i=1$. As for the algorithm in \cite{bib30}, we directly adopt the parameters specified in the original paper, namely $M_1=[0.73\ 0.84]$, $M_2=[0.68\ 0.78]$, $M_3=[0.8\ 0.98]$, $M_4=[0.65\ 0.72]$, $q_1=10$, $q_2=8$, $q_3=10$, $q_4=9$, $\alpha_1=0.3$, $\alpha_2=0.2$, $\alpha_3=0.23$, $\alpha_4=0.14$, $\beta_1=3$, $\beta_2=6$, $\beta_3=10$, $\beta_4=8$, $K_1=[0.351\ 0.056]$, $K_2=[0.344\ 0.054]$, $K_3=[0.345\ 0.052]$,  and $K_4=[0.348\ 0.055]$. 
	
	Fig. \ref{alg3x} and Fig. \ref{alg3y} show the tracking  trajectories of the agents using Algorithm 3.  
	Fig. \ref{scx} and Fig. \ref{scy} show the tracking  trajectories of the agents using algorithm in \cite{bib30}. For the first state variable, our Algorithm 3 achieves synchronization with the leader's trajectory within 15 iterations, whereas the algorithm from \cite{bib30} requires at least 126 iterations (under the condition that the maximum convergence error between the first state variable of each agent and the leader does not exceed 0.016). Regarding the second state variable, Algorithm 3 attains trajectory tracking in just 14 iterations, compared to the minimum 125 iterations needed by the algorithm in \cite{bib30} (under the condition that the maximum convergence error between the second state variable of each agent and the leader does not exceed 0.04). By comparing the results shown in the aforementioned figures, it can be observed that our algorithm enables each agent to track the leader's trajectory more rapidly. The consensus error between each agent and the leader are presented in Fig. \ref{alg3error} and Fig. \ref{scerror} (here we define the consensus error as the Euclidean distance between each agent and the leader.). The results demonstrate that Algorithm 3 achieves both faster convergence speed and better stability compared to the algorithm in \cite{bib30}. Although the consensus error of the algorithm in \cite{bib30} also exhibits a decaying trend, its convergence rate is slower with noticeable oscillations occurring within the first 50 iterations. The simulation results demonstrate that Algorithm 3 can effectively solve the leader-follower consensus problem in nonlinear multi-agent systems.
	\begin{figure}[htbp]
		\captionsetup{font={footnotesize}}
		\centering
		\includegraphics[width=0.87\linewidth]{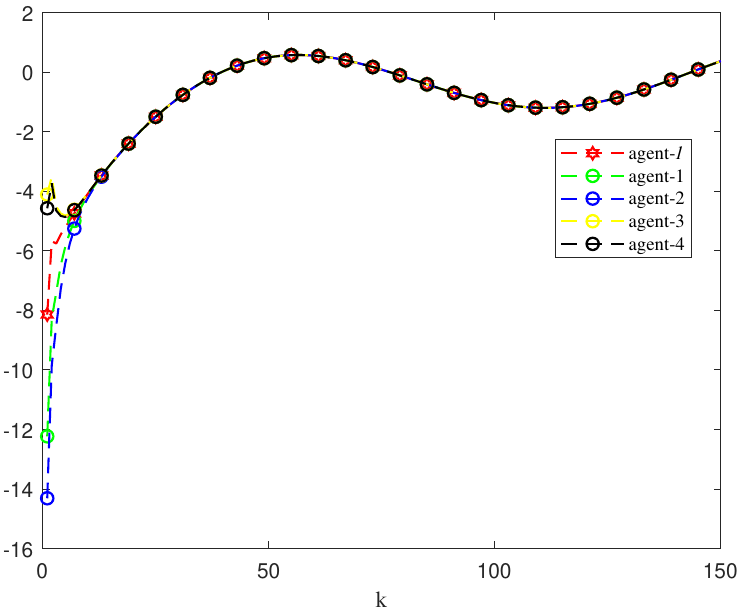}
		\caption{ The first state tracking trajectory using Algorithm 3}
		\label{alg3x}
	\end{figure}
	
	\begin{figure}[htbp]
		\captionsetup{font={footnotesize}}
		\centering
		\includegraphics[width=0.87\linewidth]{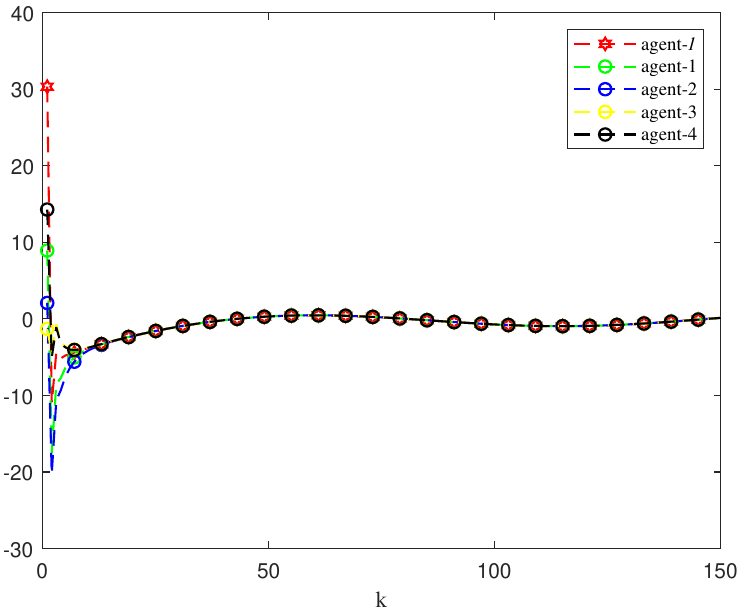}
		\caption{ The second state tracking trajectory using Algorithm 3}
		\label{alg3y}
	\end{figure}
	
	\begin{figure}[htbp]
		\captionsetup{font={footnotesize}}
		\vspace{-0.665cm}
		\centering
		\includegraphics[width=0.87\linewidth]{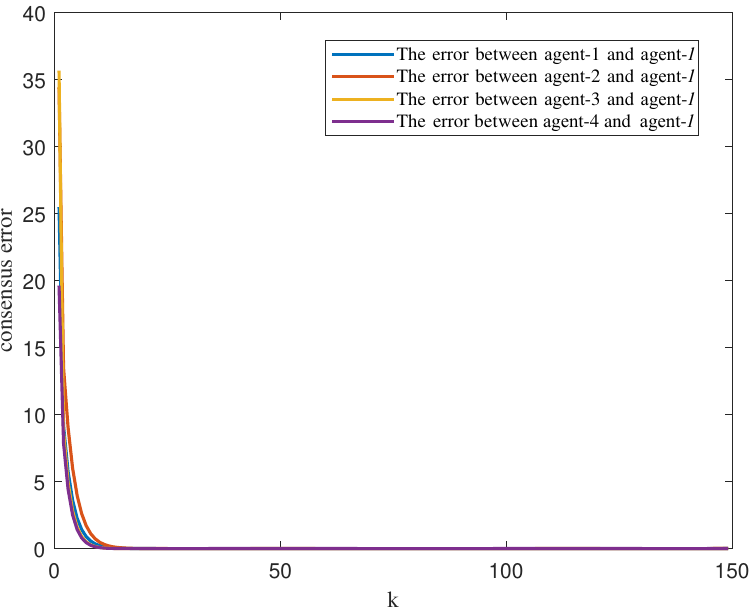}
		\caption{ Consensus errors using Algorithm 3}
		\label{alg3error}
	\end{figure}
	
	\begin{figure}[htbp]
		\captionsetup{font={footnotesize}}
		\vspace{-0.5cm}
		\centering
		\includegraphics[width=0.87\linewidth]{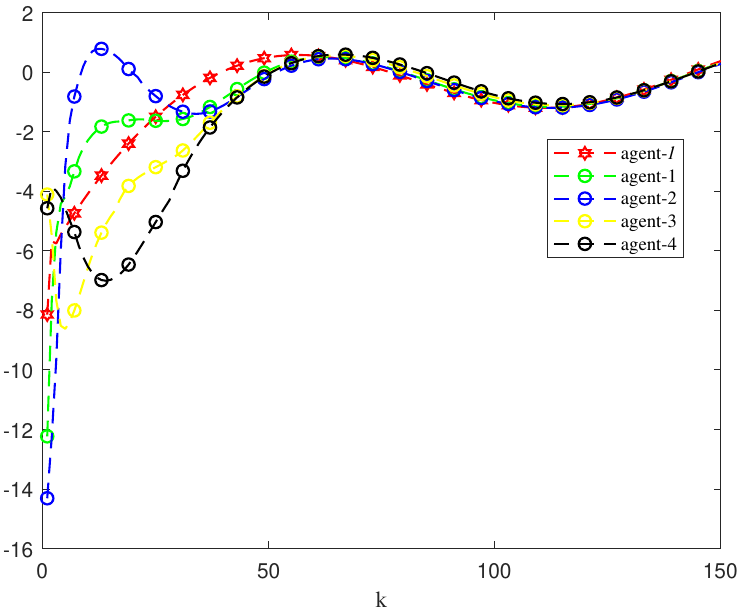}
		\caption{ The first state tracking trajectory using algorithm in \cite{bib30}}
		\label{scx}
	\end{figure}
	
	\begin{figure}[htbp]
		\captionsetup{font={footnotesize}}
		\centering
		\includegraphics[width=0.87\linewidth]{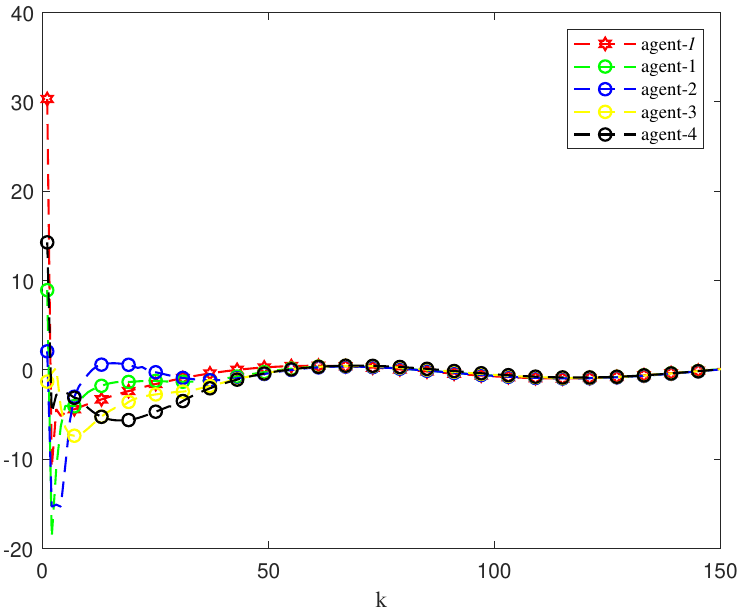}
		\caption{ The second state tracking trajectory using algorithm in \cite{bib30}}
		\label{scy}
	\end{figure}
	
	\begin{figure}[htbp]
		\captionsetup{font={footnotesize}}
		\vspace{0.1cm}
		\centering
		\includegraphics[width=0.87\linewidth]{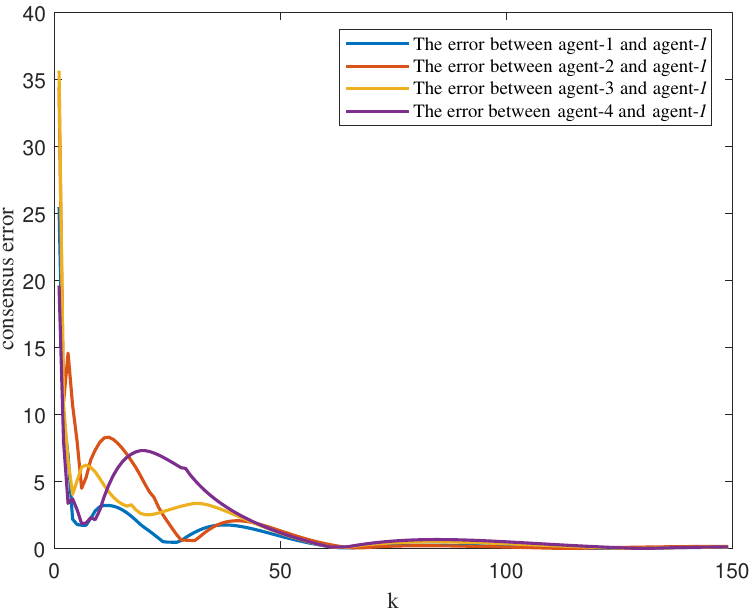}
		\caption{ Consensus errors using algorithm in \cite{bib30}}
		\label{scerror}
	\end{figure}
	
	\subsection{Formation tracking problem}
	In the preceding two examples, we have presented results for both leaderless and leader-follower scenarios, demonstrating the advantages of our algorithm. To further validate the algorithm's general applicability, we consider the following AGV formation control problem: all follower AGVs will maintain the specified formation while moving collectively with the leader AGV. To achieve this objective, it suffices to incorporate a distance-based formation function into Algorithm 3, where the neighbor position information acquired by each AGV (Step 5) inherently includes formation constraints. The communication topology for four follower AGVs and one leader AGV  is shown in Fig. \ref{ctopology3}. For notational convenience, we designate the four follower AGVs as AGV-1 through AGV-4, and denote the leader AGV as AGV-$l$.
	\begin{figure}[htbp]
		\captionsetup{font={footnotesize}}
		\centering
		\includegraphics[width=0.8\linewidth]{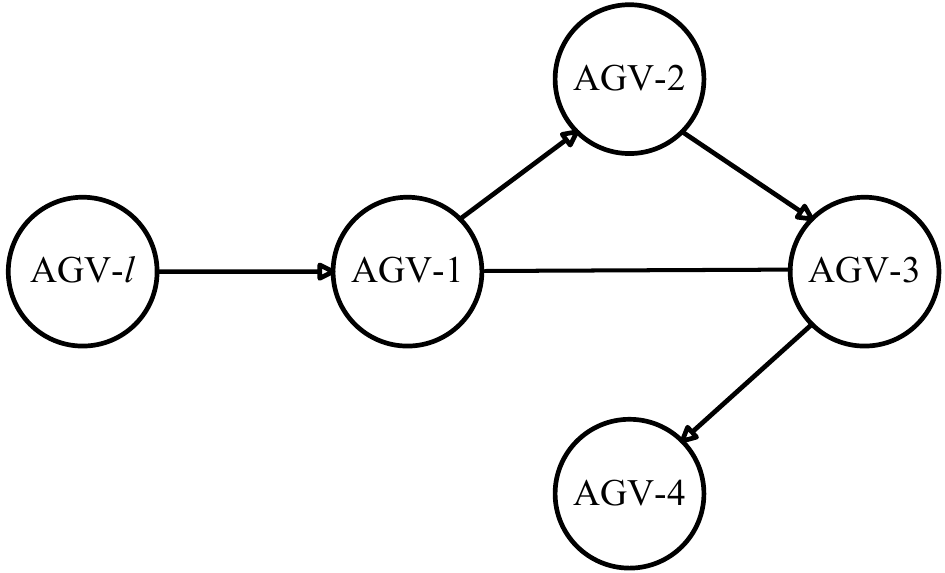}
		\caption{ Communication topology of AGVs}
		\label{ctopology3}
	\end{figure}
	
	In this example, we use Algorithm 3 to accomplish the formation task for five AGVs. The time step $\Delta=0.05$ and the prediction horizon $N_p=8$. The initial poses for AGVs are set as $\{x_l(0)=2.00,y_l(0)=0,\theta_l(0)=1.57 \}$, $\{x_1(0)=3.00,y_1(0)=0.20,\theta_1(0)=0.78\}$, $\{x_2(0)=1.20,y_2(0)=0.50,\theta_2(0)=0\}$, $\{x_3(0)=2.50,y_3(0)=0,\theta_3(0)=0.50\}$, $\{x_4(0)=1,y_4(0)=0.4,\theta_4(0)=0.8\}$ and the weight matrix $Q_{ij}=60\times I,  W_{1l}=100\times I, R_i=0.01\times I,\ i\in\{1,2,3,4\}$. In practical implementations, inter-agent communication quality is susceptible to various disturbances including environmental variations, human factors, and unexpected incidents. To explicitly evaluate these impacts, we incorporate disturbance events into the formation control simulation, where agents may experience intermittent neighbor information packet losses. This experimental design validates the algorithm's inherent disturbance resilience. As analytically established in Remark 4, the online adjustability mechanism provides robust compensation for system uncertainties and disturbances through real-time control action optimization. It should be noted that disturbance rejection is not the primary focus of this study. We only consider this simplified scenario to demonstrate the algorithm's passive disturbance resilience. For more general disturbance cases, dedicated robust control protocols would need to be designed according to the specific disturbance characteristics.
	
	Fig. \ref{ftraj} demonstrates the formation tracking trajectories of the AGVs, where the follower AGVs successfully achieve the prescribed formation and maintain coordinated motion with the leader AGV. The formation error between each AGV and its neighbors is shown in Fig. \ref{ferror}, demonstrating rapid and stable convergence to zero. The simulation results demonstrate that Algorithm 3 can effectively solve the formation problem in nonlinear multi-agent systems.
	\begin{figure}[htbp]
		\captionsetup{font={footnotesize}}
		\centering
		\includegraphics[width=0.87\linewidth]{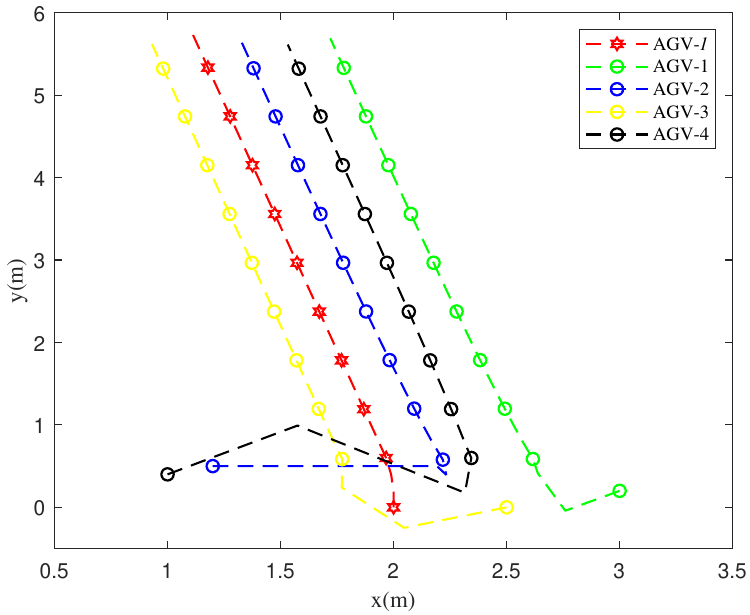}
		\caption{ Actual formation of AGVs}
		\label{ftraj}
	\end{figure}

	\begin{figure}[htbp]
		\captionsetup{font={footnotesize}}
		\centering
		\includegraphics[width=0.87\linewidth]{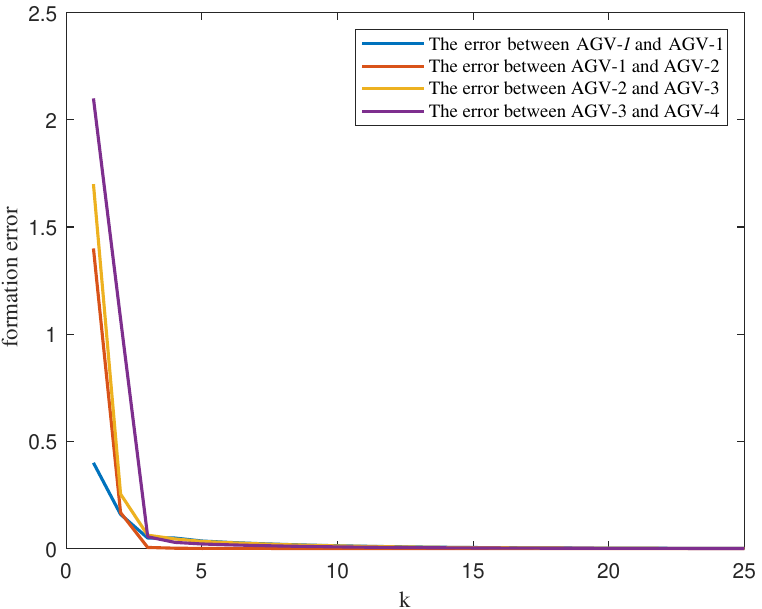}
		\caption{ Formation errors of AGVs}
		\label{ferror}
	\end{figure}
	
	\section{CONCLUSIONS}
	In this paper, we have proposed a distributed algorithm for solving the optimal consensus control problem of nonlinear multi-agent systems. The proposed algorithm exhibits fast convergence and strong stability during the iteration process. Furthermore, two enhanced algorithms have been developed under the MPC framework to solve the finite-time optimization problem in each prediction horizon, where the first part is used as the actual control input for the system. All algorithms are executed in a distributed manner and unify both the leader-follower and leaderless cases. Finally, numerical simulations has verified the effectiveness of the proposed algorithms.

\end{document}